\documentclass[11pt,reqno]{amsart}
\usepackage[T1]{fontenc}
\usepackage{lmodern}
\usepackage{amsmath,amssymb,amsthm,mathtools}
\usepackage{booktabs,array,enumitem}
\usepackage[expansion=false]{microtype}
\usepackage[hidelinks]{hyperref}
\numberwithin{equation}{section}
\setlist[enumerate]{label=\textup{(\arabic*)},leftmargin=2em,itemsep=3pt}
\newtheorem{theorem}{Theorem}[section]
\newtheorem{proposition}[theorem]{Proposition}
\newtheorem{lemma}[theorem]{Lemma}
\newtheorem{corollary}[theorem]{Corollary}
\theoremstyle{definition}

\theoremstyle{remark}
\newtheorem{remark}[theorem]{Remark}
\DeclareMathOperator{\Var}{Var}
\DeclareMathOperator{\con}{con}
\newcommand{\AI}{\mathbf{AI}}
\newcommand{\Vthree}{\mathcal V_3}
\newcommand{\Athree}{\mathcal A_3}
\newcommand{\W}{\mathcal W}
\newcommand{\F}{\mathbb F_2}
\newcommand{\eqid}{\approx}
\newcommand{\ivec}[1]{\mathbf e(#1)}
\hypersetup{
  pdftitle={The variety generated by all semirings of order three is nonfinitely based},
  pdfauthor={Aifa Wang and Lili Wang},
  pdfsubject={Nonfinite bases for varieties generated by three-element semirings}
}

\begin{document}
\title[THREE-ELEMENT SEMIRINGS AND NONFINITE BASES]
{THE VARIETY GENERATED BY ALL SEMIRINGS OF ORDER THREE IS NONFINITELY BASED}
\author[AIFA WANG]{AIFA WANG}
\address{School of Mathematical Sciences, Chongqing University of Technology,
Chongqing 400054, People's Republic of China}
\email{wangaf@cqut.edu.cn}
\thanks{Lili Wang is the corresponding author.}
\author[WENHAO JU]{WENHAO JU}
\address{School of Mathematical Sciences, Chongqing University of Technology,
Chongqing 400054, People's Republic of China}
\author[LILI WANG]{LILI WANG}
\address{School of Mathematical Sciences, Chongqing University of Technology,
Chongqing 400054, People's Republic of China}
\subjclass[2020]{08B05, 16Y60, 03C05}
\keywords{Semiring, finite basis problem, generated variety,
polynomial identity, semilattice order}
\date{}
\raggedbottom
\emergencystretch=1em

\begin{abstract}
We prove that the variety generated by all semirings of order three is
nonfinitely based, where addition is not required to be commutative and
the signature has no constants. The same conclusion holds for the
variety generated by all additively idempotent semirings of order three.
We establish these conclusions by excluding a uniform bound on the
number of variables in an identity basis. Our proof uses identities
associated with anchored odd cycles. Three small commutative test
semirings isolate a polynomial equivalence class consisting of exactly
two polynomials. For a cycle of length $n$, every first nontrivial
deduction between them requires an identity with at least $n+1$
variables, even under polynomial substitutions. A retraction followed
by a band quotient transfers absorption identities to arbitrary
addition and identifies the ai-subvariety of the full joint variety
with the joint variety of the ai-generators. Validity of the cycle
identities follows from a structural analysis of chain and flat
addition. An elementary sixth-power lemma for semigroups of order at
most three supplies the retraction. 
\end{abstract}
\maketitle

\section{Introduction}
\label{sec:intro}

The finite basis problem asks whether the identities of an algebra can
be deduced from a finite set of identities. In terms of varieties, the
problem is to determine which varieties admit a finite equational
basis. A variety is a class of algebras closed under homomorphic images,
subalgebras, and arbitrary direct products. By Birkhoff's theorem,
these are precisely the classes defined by identities; see~\cite{BS}.
We write $\Var(\mathcal K)$ for the variety generated by a class
$\mathcal K$. A variety is \emph{finitely based} if it has a finite
identity basis, and is \emph{nonfinitely based} otherwise.

Throughout this paper, a \emph{semiring} is a nonempty algebra
$(S,+,\cdot)$ in which both operations are associative and multiplication
distributes over addition on both sides. The signature contains no
constants, and additive commutativity is not assumed. An
\emph{additively idempotent semiring}, or \emph{ai-semiring}, is a
semiring whose addition is both commutative and idempotent. We denote
the variety of all ai-semirings by $\AI$, and put
\[
\begin{aligned}
\Vthree&=\Var\{S:S\text{ is a semiring and }|S|=3\},\\
\Athree&=\Var\{S:S\in\AI\text{ and }|S|=3\}.
\end{aligned}
\]
Thus $\Athree\subseteq\Vthree$. When only additive idempotence is
assumed, we shall state this explicitly.

For two-element generators, several positive results are known.
Shao and Ren~\cite{SR} proved hereditary finite basability for the
variety generated by the two-element ai-semirings. Wang, Wang, Li,
and Yin~\cite{WWLY} studied the joint variety of two-element semirings
under the convention that addition is commutative. These conventions
must be distinguished from the unrestricted additive convention used
in the definition of $\Vthree$.

For individual three-element ai-semirings, the work of Zhao
et al.~\cite{ZRCSD} and Jackson, Ren, and Zhao~\cite{JRZ} establishes
that precisely one of the $61$ isomorphism types is nonfinitely based.
This algebra is customarily denoted $S_7$. In the conclusion of
\cite[p.~130]{ZRCSD}, the authors proposed studying the variety
generated by all ai-semirings of order three. The finite basis problem
for this joint variety requires a separate argument: containing a
nonfinitely based subvariety alone does not determine the finite basis
property of a variety.

Jackson, Ren, and Zhao~\cite[Theorem~4.9 and Corollary~4.11]{JRZ}
give nonfinite basis results under additional hypotheses, including
a relative result for varieties generated by flat semirings.
Gao et al.~\cite[Theorem~2.2]{GJRZ2025} obtain a criterion using
high-girth hypergraphs and apply it to further varieties containing
$S_7$. Whether every variety generated by a finite ai-semiring and
containing $S_7$ is nonfinitely based is explicitly asked in
\cite[Problem~3.1]{GRSY2026}. Thus strong nonfinite basability of
$S_7$ is not an available premise for the joint basis problem.
Shao, Ren, and Gao~\cite{ShaoRenGao2026} give nonfinitely based
ai-semirings that are not strongly nonfinitely based, further
illustrating why this distinction matters.

We settle the joint basis problem for three-element generators,
including semirings whose addition is not assumed commutative.
We establish nonfinite basability by excluding a uniform bound on
the number of variables in an identity basis.

\begin{theorem}\label{thm:main}
Neither $\Athree$ nor $\Vthree$ has an identity basis whose identities
involve a uniformly bounded number of variables. In particular, both
varieties are nonfinitely based.
\end{theorem}

Both varieties are generated by finite algebras: take the direct
product of one representative of each three-element isomorphism type.
They are therefore locally finite. For varieties of finite signature
that are locally finite, the existence of a bounded-variable identity
basis is equivalent to finite basability; we recall the short proof
in Lemma~\ref{lem:localfinite}. The variable formulation of
Theorem~\ref{thm:main} is useful because our construction gives an
explicit lower bound for the identities needed in a deduction.

The proof combines an explicit family of valid identities with a
restriction on all possible deductions of those identities. Polynomial
substitutions and restrictions on contextual occurrences are useful
tools in ai-semiring basis problems; see~\cite{JRZ,YRG}.
Odd-cycle identities also occur in~\cite{ShaoRenGao2026}.
Here we combine an anchored cycle with
three small commutative ai-semirings to impose, respectively, restrictions
on pairs of variables, containment of supports, and affine spans over
$\F$. For a suitable polynomial $U_n$, these restrictions leave exactly
two polynomials in its equivalence class. Any deduction of the identity
between them must therefore make the entire change in its first
nontrivial step. The structure of the supports forces that step to
use at least $n+1$ variables.

The same obstruction proves that the $18$-element direct product of
the three test semirings is nonfinitely based, although each factor
is finitely based. For the passage to arbitrary addition, we establish
the structural identity
$\Vthree\cap\AI=\Athree$. It reduces the final step of the main
theorem to adjoining the two identities for semilattice addition.

We organize the proof as follows. Section~\ref{sec:prelim} records the
polynomial notation and the form of equational deductions.
Section~\ref{sec:tests} introduces the three test semirings.
In Section~\ref{sec:cycles} we define the anchored odd-cycle identities
and prove their validity by analyzing the two possible shapes of
three-element additive semilattices. Sections~\ref{sec:isolation}
and~\ref{sec:nfb} prove the
isolation and the unbounded-variable obstruction, completing the
argument for $\Athree$. Section~\ref{sec:arbitrary} transfers the
identity family to $\Vthree$ and completes the proof of
Theorem~\ref{thm:main}.

\section{Preliminaries}
\label{sec:prelim}

We first record the relationship between the two basis formulations
used in the introduction.

\begin{lemma}\label{lem:localfinite}
Let $\mathcal V$ be a locally finite variety of finite signature.
Then $\mathcal V$ has an identity basis with a uniform finite
variable bound if and only if it is finitely based.
\end{lemma}

\begin{proof}
A finite basis has a largest variable count. Conversely, suppose
that $\Sigma$ is a basis whose identities have at most $k$ variables,
where we may take $k\geq1$. Rename the variables of each identity
into $x_1,\ldots,x_k$. The free algebra $F=F_{\mathcal V}(k)$ is
finite. For each $a\in F$, choose a representative term $t_a$
in these variables. Write $[t]$ for the element of $F$ represented
by a term $t$. Let $\Delta$ consist of
\[
x_i\eqid t_{[x_i]}\quad(1\leq i\leq k)
\]
and, for each basic operation $f$ of arity $m$ and each tuple
$(a_1,\ldots,a_m)\in F^m$, the identity
\[
f(t_{a_1},\ldots,t_{a_m})
\eqid t_{f^F(a_1,\ldots,a_m)}.
\]
This includes the analogous identity for a nullary operation if
one is present. The set $\Delta$ is finite, and its identities
hold in $\mathcal V$. Induction on terms shows that
$\Delta\vdash t\eqid t_{[t]}$ for every term in $x_1,\ldots,x_k$.
Consequently, $\Delta$ derives every identity of $\mathcal V$
in these variables, and hence every renamed member of $\Sigma$.
Thus $\operatorname{Mod}(\Delta)=\mathcal V$, proving that
$\Delta$ is a finite basis.
\end{proof}

The semiring axioms used throughout are
\[
\begin{aligned}
(x+y)+z&\eqid x+(y+z), & (xy)z&\eqid x(yz),\\
x(y+z)&\eqid xy+xz, & (x+y)z&\eqid xz+yz.
\end{aligned}
\]
For an ai-semiring $S$, the relation
\[
a\leq b\quad\Longleftrightarrow\quad a+b=b
\]
is its additive semilattice order. We write $q\preceq u$ for the
identity $u+q\eqid u$. Equivalently, every evaluation sends the value
of $q$ below the value of $u$ in this order.

In a \emph{commutative} ai-semiring, multiplication is also commutative.
Let $X$ be a countably infinite set of variables and let $X_c^+$ be the
free commutative semigroup on $X$. A word in $X_c^+$ is nonempty and is
determined by the multiplicity of each variable. A \emph{polynomial}
is a finite nonempty subset of $X_c^+$, written as the formal sum of
its words. Addition of polynomials is union, and their product is
setwise word multiplication. These operations give the free
commutative ai-semiring on $X$.

We distinguish equality of polynomials, written $=$, from an identity,
written $\eqid$. In particular, repeated summands disappear in a
polynomial, but repeated letters in a word do not. For a word or a
polynomial $u$, its \emph{content} $\con(u)$ is the set of variables
occurring in it. A word is \emph{linear} if each of its variables occurs
exactly once; a polynomial is linear if all of its words are linear.
A substitution sends each variable to a nonempty polynomial and
extends to a homomorphism of the free commutative ai-semiring.

An empty word or an empty polynomial is not a term in our signature.
We use them only to indicate, respectively, an absent multiplicative
context or an absent additive context in a deduction. In particular,
substitutions never take empty values.

\begin{lemma}[Absorption form]\label{lem:absorption}
Modulo the axioms of commutative ai-semirings, every identity is
equivalent to a finite set of identities $s+r\eqid s$, where $s$ is
a polynomial and $r$ is a word. This replacement introduces no
additional variables. Every contextual application of an absorption
identity has the form
\begin{equation}\label{eq:step}
K\varphi(s)+R
\quad\longleftrightarrow\quad
K\varphi(s)+K\varphi(r)+R.
\end{equation}
Here $\varphi$ is a substitution, $K$ is a polynomial or an absent
multiplicative context, and $R$ is a polynomial or an absent additive
context.
\end{lemma}

\begin{proof}
Write the normalized sides of an identity as
$u=u_1+\cdots+u_t$ and $v=v_1+\cdots+v_m$, where each $u_i$ and
$v_j$ is a word. Replace $u\eqid v$ by
\[
v+u_i\eqid v\quad(1\leq i\leq t),
\qquad
u+v_j\eqid u\quad(1\leq j\leq m).
\]
If $u\eqid v$ holds, then every $u_i$ is below $u=v$ and every
$v_j$ is below $v=u$, so all the displayed identities hold.
Conversely, their conjunction gives
$u+v\eqid v$ and $u+v\eqid u$, and hence $u\eqid v$.
All words used in this replacement come from the original sides,
which proves the assertion about variables.

For the contextual assertion, consider a term context with exactly
one distinguished occurrence of its argument $t$. We prove by induction
on the construction of the context that its polynomial form is $Kt+R$.
For the context consisting only of $t$, both contexts are absent.
Adding a term $A$ to $Kt+R$ changes the additive context to $R+A$.
Multiplying by $A$ gives
\[
A(Kt+R)=(AK)t+AR.
\]
The same argument applies to multiplication on the right because
multiplication is commutative. Thus the asserted form is preserved
at every stage. Substituting $\varphi(s)$ and
$\varphi(s)+\varphi(r)$ for the distinguished argument now gives
\eqref{eq:step}. If either context is absent, the corresponding
factor or summand is simply omitted.
\end{proof}

By the standard description of equational consequence, an identity
follows from a set $\Sigma$ of identities precisely when its two sides
are connected by a finite sequence of contextual substitutions of
identities from $\Sigma$, in either direction; see~\cite{BS}.
After the replacement in Lemma~\ref{lem:absorption}, each step between
normalized polynomials has form~\eqref{eq:step}. The axioms of
commutative ai-semirings are already built into the polynomial
representation. This remains true when $\Sigma$ is infinite, since
each particular deduction is finite.

\section{Three test semirings}
\label{sec:tests}

Let $D=\{0,1\}$ be the two-element distributive lattice, with join as
addition and meet as multiplication. Let $H=\{1,a,\infty\}$ have
additive chain $1<a<\infty$, multiplicative identity $1$, and
$a^2=\infty$, with $\infty$ absorbing for multiplication. Finally,
let $G=\{1,g,\infty\}$ be the flat extension of the cyclic group of
order two. Thus $g^2=1$, the element $\infty$ is multiplicatively
absorbing, and any two distinct elements have sum $\infty$.
The operation tables of $H$ and $G$ are displayed in
Table~\ref{tab:tests}. All element names are labels, not constant
symbols in the signature.

\begin{table}[ht]
\caption{The test semirings $H$ and $G$.}
\label{tab:tests}
\centering
\setlength{\tabcolsep}{5pt}
\renewcommand{\arraystretch}{1.1}
\begin{tabular}{c@{\qquad}c}
$\begin{array}{c|ccc}
+_H&1&a&\infty\\\hline
1&1&a&\infty\\
a&a&a&\infty\\
\infty&\infty&\infty&\infty
\end{array}$&
$\begin{array}{c|ccc}
\cdot_H&1&a&\infty\\\hline
1&1&a&\infty\\
a&a&\infty&\infty\\
\infty&\infty&\infty&\infty
\end{array}$\\
\noalign{\vskip12pt}
$\begin{array}{c|ccc}
+_G&1&g&\infty\\\hline
1&1&\infty&\infty\\
g&\infty&g&\infty\\
\infty&\infty&\infty&\infty
\end{array}$&
$\begin{array}{c|ccc}
\cdot_G&1&g&\infty\\\hline
1&1&g&\infty\\
g&g&1&\infty\\
\infty&\infty&\infty&\infty
\end{array}$
\end{tabular}
\end{table}

These algebras are commutative ai-semirings. For $H$, multiplication
on the chain is associative and order preserving; every order-preserving
map on a chain preserves binary joins, which gives distributivity.
For $G$, multiplication by a group element permutes the two group
elements and fixes $\infty$, hence preserves the flat join. Multiplication
by $\infty$ is constant, and also preserves it.

Set $\W=\Var(H,D,G)$. Both $H$ and $G$ have order three.
To place $D$ in $\Athree$, adjoin a new element $\omega$ absorbing
for both operations. All axioms on old elements remain valid, and an
axiom involving $\omega$ has value $\omega$ on both sides. The resulting
three-element ai-semiring contains $D$. Therefore
\begin{equation}\label{eq:testsinside}
\W\subseteq\Athree\subseteq\Vthree.
\end{equation}
The algebra $H$ is usually denoted $S_{53}$ in the three-element
classification; it differs from the flat nilpotent example $S_7$.

For a linear word $w$ over a finite variable set $Y$, write
$\ivec{w}\in\F^Y$ for its incidence vector. Its $x$-coordinate is $1$
when $x\in\con(w)$ and is $0$ otherwise. Recall that the affine span
of vectors $e_1,\ldots,e_t$ over $\F$ consists of sums
$\lambda_1e_1+\cdots+\lambda_te_t$ with
$\lambda_1+\cdots+\lambda_t=1$.

\begin{lemma}\label{lem:tests}
Let $u$ be a linear polynomial and let $r$ be a commutative word.
If $\W\models r\preceq u$, then the following conditions hold:
\begin{enumerate}
\item $r$ is linear and $\con(r)\subseteq\con(u)$;
\item every two distinct variables of $r$ occur together in a word of $u$;
\item some word of $u$ has content contained in $\con(r)$;
\item $\ivec{r}$ belongs to the affine span over $\F$ of the incidence
vectors of the words of $u$.
\end{enumerate}
\end{lemma}

\begin{proof}
Since $H,D,G$ belong to $\W$, the inequality must hold in each of
these three algebras. We prove the four assertions in order.

\emph{(1)} Suppose first that a variable $x$ of $r$ does not occur
in $u$. In $H$, assign $\infty$ to $x$ and $1$ to all other variables.
Then $r$ has value $\infty$ and $u$ has value $1$, contradicting
$r\leq u$. Now suppose that $x$ is repeated in $r$. Assign $a$ to
$x$ and $1$ to all other variables. The value of $r$ is $\infty$
because $a^2=\infty$. Every word of $u$ is linear, so its value is
$1$ or $a$; the value of their sum is therefore at most $a$.
This again contradicts the inequality.

\emph{(2)} Let $x,y$ be distinct variables of $r$ and suppose that
no word of $u$ contains both. In $H$, assign $a$ to $x,y$ and $1$
to the other variables. The value of $r$ is $\infty$, while every
word of $u$ has value at most $a$. Thus such a pair cannot occur.

\emph{(3)} In $D$, assign $1$ to the variables of $r$ and $0$ to
all the others. Then $r$ has value $1$. A word of $u$ has value $1$
exactly when all its variables occur in $r$. Since addition in $D$
is join and $r\leq u$, at least one word of $u$ must have this property.

\emph{(4)} Work in the vector space on $\con(u)$, which contains
$\con(r)$ by (1). Let $e_1,\ldots,e_t$ be the incidence vectors of
the words of $u$, and put $d=\ivec{r}$. Affine-span membership of
$d$ is equivalent to
\[
(d,1)\in\operatorname{span}_{\F}\{(e_1,1),\ldots,(e_t,1)\}.
\]
If this membership fails, there is a linear functional that vanishes
on the displayed span and takes value $1$ on $(d,1)$: choose a basis
of the span, extend it by $(d,1)$ to a basis, and prescribe the
functional on that basis. Writing it as the scalar product with
$(b,c)$ gives
\[
e_j\cdot b+c=0\quad(1\leq j\leq t),
\qquad d\cdot b+c=1.
\]
In $G$, assign $g^{b_x}$ to each variable $x$. Each word of $u$
then has value $g^c$, so their sum also has value $g^c$.
The value of $r$ is $g^{c+1}$. These are the two distinct group
elements of $G$, whose sum is $\infty$; neither is below the other.
This contradicts $r\preceq u$ and proves (4).
\end{proof}

Only the indicated test algebra is needed for each support
condition: $H$ gives (1) and (2), and $D$ gives (3).
We shall apply these conditions both to the cycle variables and
to the formal variables occurring in a proposed basis identity.

\begin{remark}\label{rem:hypergraphs}
The support conditions explain why the hypergraph criterion in
\cite[Theorem~2.2]{GJRZ2025} does not apply directly to $\Athree$.
Let $\mathbb H$ be a $3$-uniform hypergraph of girth at least five,
and let $t_{\mathbb H}$ be its edge polynomial. If a term $w$ were
absorbed by $t_{\mathbb H}$ in $\Athree$, then Lemma~\ref{lem:tests}(1)--(3),
applied in $H$ and $D$, would force the support of each word of $w$
to be a clique containing an edge of $\mathbb H$.
Such a clique is exactly that edge. Indeed, if an additional
vertex $v$ were adjacent to two vertices $a,b$ of an edge $e$,
choose edges $e_a,e_b$ containing $\{v,a\}$ and $\{v,b\}$.
If $e_a=e_b$, this edge and $e$ form a hypergraph cycle of length
two; otherwise $e,e_a,e_b$ form one of length three.
Both contradict the girth condition.

Thus every word of $w$ is an edge word, up to permutation of its
letters. Under the defining vertex assignment in the commutative
hypergraph semiring of~\cite{GJRZ2025}, all such words take the
same edge value, and so does their sum. Therefore $w$ cannot be
a non-hyperedge term, as required by that criterion.
The family below instead combines words of different lengths
with an anchor.
\end{remark}

\section{Anchored odd-cycle identities}
\label{sec:cycles}

Let $n\geq3$ be odd and choose distinct variables
$x_1,\ldots,x_n,y_1,\ldots,y_n,z$. Indices on the $y$-variables are
read modulo $n$, so $y_{n+1}=y_1$. Define
\begin{equation}\label{eq:cycle}
\begin{aligned}
p_n&=x_1x_2\cdots x_n,&
a_n&=x_1z,&
q_n&=p_nz,\\
b_i&=x_i y_i y_{i+1}z\quad(1\leq i\leq n),&
U_n&=p_n+a_n+\sum_{i=1}^{n}b_i.
\end{aligned}
\end{equation}
The words $b_i$ record the edges $y_i y_{i+1}$ of an odd cycle.
The word $a_n$ supplies an anchor at $x_1$. At this stage these are
ordinary semiring terms: products have their displayed order, and
the sum of the $b_i$ is taken in increasing index order. Put
\[
\sigma_n:\quad U_n+q_n\eqid U_n.
\]
We first give an elementary proof of validity on the test variety.

\begin{lemma}\label{lem:test-valid}
The variety $\W$ satisfies $\sigma_n$ for every odd $n\geq3$.
\end{lemma}

\begin{proof}
It suffices to check the identity in $D,H,G$. In $D$, multiplication
is meet, so $q_n=p_nz\leq p_n\leq U_n$.

In $H$, the value of $q_n$ lies in $\{1,a,\infty\}$.
If it is $1$, the inequality is automatic because $1$ is the least
element. If it is $a$, exactly one of $x_1,\ldots,x_n,z$ has
value $a$ and the others have value $1$. If this variable is an
$x_i$, then $p_n=a$; if it is $z$, then $a_n=a$.
In either case $U_n\geq a$.
Suppose that $q_n=\infty$. If an $x_i$ has value $\infty$,
then $p_n=\infty$; if $z=\infty$, then $a_n=\infty$.
Otherwise at least two of the variables $x_1,\ldots,x_n,z$
have value $a$. Two such $x$-variables give $p_n=\infty$.
The remaining possibility is $x_i=z=a$ for some $i$, which
gives $b_i=\infty$ regardless of the $y$-values. Thus
$U_n=\infty$ whenever this is required.

In $G$, a variable assigned $\infty$ makes a word of $U_n$
equal to $\infty$, so the inequality holds. We may therefore
assume that all variables take values in the group $\{1,g\}$.
If $U_n=\infty$, there is again nothing to prove. Otherwise
all its summands take one common group value $c$, since
addition is flat. In particular, $b_i=c$ for every $i$.
Evaluating in this commutative group gives
\[
\prod_{i=1}^{n}b_i
=\left(\prod_{i=1}^{n}x_i\right)
 \left(\prod_{i=1}^{n}y_i^2\right)z^n
=p_nz=q_n.
\]
On the other hand, $\prod_i b_i=c^n=c$, since $n$ is odd
and the group has exponent two. Hence $q_n=c=U_n$.
The identity holds in each generator, and therefore in $\W$.
\end{proof}

We next prove validity on every three-element ai-semiring. The
additive semilattice is either a chain or consists of two incomparable
elements and their join. We treat these shapes separately.

\begin{lemma}\label{lem:chain-valid}
Let $S$ be an ai-semiring of order at most three whose additive order
is a chain. Then $S$ satisfies $U_n+q_n\eqid U_n$ for every $n\geq3$,
with the terms defined in~\eqref{eq:cycle}.
\end{lemma}

\begin{proof}
Multiplication is order preserving in each argument, by distributivity.
Fix an evaluation, and use the variable and term symbols also for
their values. Suppose, towards a contradiction, that $q_n>U_n$.
Write $P=x_1\cdots x_n$. Since
\[
P\leq U_n<Pz,\qquad x_1z\leq U_n<Pz,
\]
monotonicity and the chain order give $x_1<P<Pz$. This is impossible
if $|S|\leq2$. Otherwise label its elements $0<1<2$; these are order
labels, not designated constants of the semiring. Necessarily
\begin{equation}\label{eq:chain-counterexample}
x_1=0,\qquad P=U_n=1,\qquad 1z=2,\qquad 0z\leq1.
\end{equation}
For each $i$, monotonicity gives
\[
(x_i00)z\leq x_i y_i y_{i+1}z=b_i\leq1.
\]
If $x_i00\geq1$, then the left side is at least $1z=2$.
Thus $x_i00=0$ for every $i$. Taking $i=1$ gives $000=0$.
As $0\leq00$, right multiplication by $0$ yields
$00\leq000=0$, so $00=0$. It follows that $x_i0=0$ for all $i$.

We claim that, under~\eqref{eq:chain-counterexample}, every element
$t$ with $t0=0$ also satisfies $0t=0$. Suppose instead that
$w=0t>0$. Then $t\geq1$, since $00=0$, and monotonicity gives
$10\leq t0=0$. Associativity now gives
\[
10=0,\qquad w^2=0(t0)t=00t=w,\qquad 0w=00t=w.
\]
If $w=1$, these equations imply $11=01=1$. The equations
$10=0$ and $1z=2$ force $z=2$, and hence $12=2$.
But then
\[
02=0(12)=(01)2=12=2,
\]
contrary to $0z\leq1$.

If $w=2$, then $02=2$. The inequality $0z\leq1$ excludes $z=2$,
and $10=0$ together with $1z=2$ excludes $z=0$. Therefore
$z=1$, $11=2$, and $01\leq1$. By associativity,
\[
(01)1=0(11)=02=2.
\]
The value $01=0$ is impossible in this equality, so $01=1$.
This gives the contradiction
\[
(10)1=01=1\ne2=11=1(01).
\]
The claim follows. Applying it to each $x_i$ gives $0x_i=0$.
Since $x_1=0$, we obtain $P=x_1\cdots x_n=0$, contradicting
$P=1$ in~\eqref{eq:chain-counterexample}.
\end{proof}

The next lemma reduces one part of the non-chain case to the chain case.

\begin{lemma}\label{lem:flat-quotients}
Let $S$ be a three-element ai-semiring whose additive semilattice
consists of incomparable elements $a,b$ and their join $\top$.
If $\top$ is not a two-sided multiplicative zero, then $S$ embeds
in a direct product of two ai-semirings of order two.
\end{lemma}

\begin{proof}
Suppose first that $\top y$ is a minimal element for some $y\in S$,
and name this element $b$. For every $s\in S$, monotonicity gives
$sy\leq\top y=b$. Since $b$ is minimal, $sy=b$.
Consequently, for every $x\in S$,
\[
xb=x(\top y)=(x\top)y=b.
\]
Thus $b$ is a right zero for multiplication. If instead an element
$x\top$ is minimal, the same argument in the opposite semiring
gives a left zero. It suffices to treat the right-zero case:
the additive relations constructed below are also congruences for
the original multiplication whenever they are congruences for its
opposite.

Consider the two equivalence relations
\[
\theta_a=\bigl\{\{a,\top\},\{b\}\bigr\},\qquad
\theta_b=\bigl\{\{b,\top\},\{a\}\bigr\}.
\]
Both are additive congruences. We show that every left and right
multiplication preserves them. A join-preserving map on this
semilattice whose value at $\top$ is minimal is constant.
For a nonconstant join-preserving map, the value at $\top$ is
$\top$; such a map preserves both displayed relations if it sends
each of $a,b$ either to itself or to $\top$.

Left multiplication by $x$ fixes $b$, since $xb=b$. If it is
nonconstant, it fixes $\top$, and its value at $a$ cannot be $b$:
otherwise $x\top=xa+xb=b$. Hence every left multiplication
preserves both relations.

For right multiplication by $y$, we show that any map sending
one minimal element to the other must be constant. Right multiplication
by $b$ is already constant. Suppose that $ay=b$. If $y=a$,
then $aa=b$, and associativity gives
\[
ba=(aa)a=a(aa)=ab=b.
\]
Thus $\top a=aa+ba=b$, so right multiplication by $a$ is
constant. If $y=\top$, then
$b=a\top=aa+ab=aa+b$ forces $aa=b$. The preceding calculation
gives $ba=b$, and therefore $b\top=ba+bb=b$. Together with
$a\top=b$, this shows that right multiplication by $\top$
is constant.

Now suppose that $by=a$. The cases $y=b$ and $y=\top$ are
impossible: $bb=b$, and $b\top=ba+bb\geq b$.
If $y=a$, then $ba=a$, and
\[
a=(ab)a=a(ba)=aa.
\]
It follows that
$\top a=aa+ba=a$, and right multiplication by $a$ is constant.
We have excluded every nonconstant map sending $a$ to $b$ or $b$ to $a$.
Every right multiplication therefore preserves both relations.

Thus $\theta_a$ and $\theta_b$ are semiring congruences.
Their intersection is equality, so
\[
S\longrightarrow S/\theta_a\times S/\theta_b,\qquad
s\longmapsto([s]_{\theta_a},[s]_{\theta_b}),
\]
is an injective homomorphism. Each quotient has two elements,
and its addition is a chain semilattice.
\end{proof}

\begin{proposition}\label{prop:valid}
Every ai-semiring of order at most three satisfies $\sigma_n$
for every odd $n\geq3$. Consequently,
$\Athree\models\sigma_n$ for all such $n$.
\end{proposition}

\begin{proof}
The chain case follows from Lemma~\ref{lem:chain-valid}.
Every remaining additive semilattice has three elements and is
of the form in Lemma~\ref{lem:flat-quotients}. If its top element
is not a two-sided multiplicative zero, that lemma embeds the
semiring into a product of two chain ai-semirings of order two.
The identity holds in both factors by Lemma~\ref{lem:chain-valid},
and hence in the given semiring.

It remains to consider the case in which the additive top is
a two-sided multiplicative zero. In this part of the proof denote
it by $0$, and write $S=\{a,b,0\}$, where $a+b=0$.
The symbol $0$ thus denotes the greatest additive element here.
Distributivity implies cancellation whenever a product is nonzero:
\begin{equation}\label{eq:zero-cancellation}
xy=xz\ne0\ \Longrightarrow\ y=z,\qquad
yx=zx\ne0\ \Longrightarrow\ y=z.
\end{equation}
Indeed, if $y\ne z$ then $y+z=0$, so
$xy+xz=x(y+z)=x0=0$, contradicting $xy=xz\ne0$.
The other implication follows in the same way.

We distinguish the number of nonzero multiplicative idempotents.

\emph{No nonzero idempotent.}
Every element is nilpotent, because its finite cyclic subsemigroup
contains an idempotent, which must be $0$. We have $ab=ba=0$.
To prove the first equality, $ab=a$ would imply $ab^k=a$ for
every $k\geq1$, contradicting nilpotence of $b$; and $ab=b$
would imply $a^k b=b$, contradicting nilpotence of $a$.
The proof for $ba$ is identical. Also $a^2\in\{b,0\}$ and
$b^2\in\{a,0\}$, so $a^3=b^3=0$. Any product of three
elements is therefore zero: a word containing both $a$ and $b$
has an adjacent mixed pair, and a word with only one nonzero
letter has cube zero. Since $n\geq3$, we have $p_n=0$, hence
$U_n=0$, and $\sigma_n$ follows.

\emph{Two nonzero idempotents.}
Write them as $e,f$. By~\eqref{eq:zero-cancellation},
$ef=e$ would imply $ee=ef\ne0$ and $e=f$; similarly $ef=f$
would imply $ef=ff\ne0$ and $e=f$. Thus $ef=fe=0$.
Fix an evaluation. If $U_n=0$, the required inequality is
automatic. Otherwise all its summands have the same nonzero
value $c$. A nonzero product consists entirely of one of
the idempotents, so $p_n=c$ forces $x_i=c$ for every $i$.
Then $a_n=x_1z=c$ forces $z=c$, and consequently $q_n=c$.

\emph{Exactly one nonzero idempotent.}
Write $S=\{e,t,0\}$ with $e^2=e$ and $t^2\ne t$.
By~\eqref{eq:zero-cancellation}, neither $et$ nor $te$ can
equal $e$, so $et,te\in\{t,0\}$. Also $t^2\in\{e,0\}$.
If $t^2=e$, then $t^4=e$, and
$t^3=et=te\in\{t,0\}$. The value $t^3=0$ would force
$t^4=0$, so $et=te=t$. Thus $\{e,t\}$ is the group of
order two, and $S$ is isomorphic to $G$. Its odd-cycle
identity was proved in Lemma~\ref{lem:test-valid}.

Suppose finally that $t^2=0$. A product containing $0$ is
zero, and a product containing at least two occurrences of
$t$ is also zero. For the latter assertion, take two
consecutive occurrences of $t$. The intervening letters,
if any, are $e$; as $e^2=e$ and $te\in\{t,0\}$, their
product with the two occurrences of $t$ is zero. Hence a
nonzero product has either only $e$'s, with value $e$, or
exactly one $t$, with value $t$.

Again we may assume that $U_n=c\ne0$. If $c=e$, then
$p_n=e$ forces every $x_i=e$, and $a_n=e$ forces $z=e$.
Thus $q_n=e$. If $c=t$, then $p_n=t$ forces exactly one
$x_i=t$ and all the others to equal $e$. The value $z=0$
is impossible. If $z=t$, the edge word $b_i$ corresponding
to $x_i=t$ contains two occurrences of $t$, so $b_i=0$,
a contradiction. Therefore $z=e$. Now $a_n=x_1e=t$ forces
$x_1=t$ and $te=t$, so $q_n=p_ne=te=t$.

These cases exhaust all possibilities and establish $\sigma_n$
on every generator of $\Athree$. Preservation of identities
under homomorphic images, subalgebras, and direct products
completes the proof.
\end{proof}

\begin{remark}
The oddness assumption is essential. In $G$, set $n=4$ and assign
\[
(x_1,x_2,x_3,x_4)=(g,g,1,1),\qquad
(y_1,y_2,y_3,y_4)=(1,1,1,g),\qquad z=g.
\]
Then $p_4=a_4=b_1=b_2=b_3=b_4=1$, whereas $q_4=g$.
Thus $U_4+q_4=\infty\ne1=U_4$.
\end{remark}

\section{An isolated polynomial equivalence class}
\label{sec:isolation}

From now until the end of Section~\ref{sec:nfb}, the terms in
\eqref{eq:cycle} are regarded as commutative polynomials.
This is justified by the use of the commutative test variety $\W$.
No multiplicative commutativity assumption is made on $\Athree$
or $\Vthree$.

\begin{lemma}[Isolation]\label{lem:isolation}
Let $n\geq5$ be odd. If $\W\models r\preceq U_n$ for a
commutative word $r$, then
\[
r\in\{p_n,a_n,b_1,\ldots,b_n,q_n\}.
\]
Moreover, the equivalence class of $U_n$ under the identities of
$\W$, in the free commutative ai-semiring, is exactly
$\{U_n,U_n+q_n\}$.
\end{lemma}

\begin{proof}
We first determine the possible words below $U_n$ and then show
that every word originally present in $U_n$ is indispensable.

By Lemma~\ref{lem:tests}, $r$ is linear and uses only variables of
$U_n$. Form a graph $\Gamma_n$ with these variables as vertices,
joining two distinct vertices when they occur together in a word
of $U_n$. Conditions (2) and (3) of that lemma say that
$\con(r)$ is a clique containing the content of at least one word
of $U_n$.

The $x$-vertices form a clique, and $z$ is adjacent to every
$x$- and $y$-vertex. The only $y$-neighbours of $x_i$ are
$y_i,y_{i+1}$. Among the $x$-vertices, the neighbours of $y_i$
are precisely $x_{i-1},x_i$; among the $y$-vertices, its neighbours
are $y_{i-1},y_{i+1}$. We use this description in the following
three cases.

\emph{Case 1: $\con(p_n)\subseteq\con(r)$.}
No $y_i$ is adjacent to every $x$-vertex, since $n\geq5$ and
$y_i$ has only two $x$-neighbours. Hence the only additional
vertex that can occur in the clique $\con(r)$ is $z$.
Linearity now gives $r=p_n$ or $r=q_n$.

\emph{Case 2: $\con(b_i)\subseteq\con(r)$ for some $i$.}
The common $x$-neighbour of $y_i,y_{i+1}$ is only $x_i$.
Also, no additional $y$-vertex is adjacent to both $y_i$ and
$y_{i+1}$ in a cycle of length at least five. Thus the clique
$\{x_i,y_i,y_{i+1},z\}$ is maximal, and $r=b_i$.

\emph{Case 3: $\{x_1,z\}\subseteq\con(r)$.}
Any $y$-variables in $r$ must lie in $\{y_1,y_2\}$, because
they must be adjacent to $x_1$. Every incidence vector of a
word of $U_n$ has an even sum of its $y$-coordinates.
This property is preserved by linear combinations over $\F$,
so Lemma~\ref{lem:tests}(4) gives the same property for
$\ivec{r}$. Thus either both $y_1,y_2$ occur or neither occurs.
If both occur, their only common $x$-neighbour is $x_1$,
and no further $y$-variable is allowed. Hence $r=b_1$.

Suppose that no $y$-variable occurs in $r$. By affine-span
membership, there are $\alpha,\beta,\gamma_1,\ldots,\gamma_n\in\F$
such that
\begin{equation}\label{eq:affine}
\begin{aligned}
\ivec{r}
 &=\alpha\ivec{p_n}+\beta\ivec{a_n}
       +\sum_{i=1}^{n}\gamma_i\ivec{b_i},\\
\alpha+\beta+\sum_{i=1}^{n}\gamma_i&=1.
\end{aligned}
\end{equation}
The $y_i$-coordinate on the left is zero. On the right, it is
$\gamma_{i-1}+\gamma_i$, with $\gamma_0=\gamma_n$.
Thus
\[
\gamma_1=\gamma_2=\cdots=\gamma_n=\gamma.
\]
Since $n$ is odd, $\sum_i\gamma_i=\gamma$ in $\F$.
Since $z$ occurs in $r$, comparison of the $z$-coordinates
and the coefficient condition in \eqref{eq:affine} give
\[
\beta+\gamma=1,\qquad \alpha+\beta+\gamma=1.
\]
It follows that $\alpha=0$. If $(\beta,\gamma)=(1,0)$,
then $\ivec{r}=\ivec{a_n}$ and $r=a_n$.
If $(\beta,\gamma)=(0,1)$, then
$\ivec{r}=\sum_i\ivec{b_i}$.
In this sum every $y_i$ occurs twice, every $x_i$ once, and
$z$ occurs an odd number of times. Hence this vector is
$\ivec{q_n}$ and $r=q_n$. This completes the list of possible
words, because condition (3) of Lemma~\ref{lem:tests} ensures
that at least one of the three cases applies.

Now let $\W\models v\eqid U_n$. Each word of $v$ is below
$v$ and therefore below $U_n$, so every word of $v$ belongs
to the displayed list. We claim that $p_n,a_n,b_1,\ldots,b_n$
must all occur in $v$.

If $p_n$ were absent, assign $1$ to all $x_i$ and $0$ to
$z$ and all $y_i$ in $D$. Then $U_n$ has value $1$,
whereas every remaining listed word contains $z$ and has value
$0$. If $a_n$ were absent, assign $1$ only to $x_1,z$ in
$D$. The word $a_n$ has value $1$, but every other listed
word contains a variable assigned $0$. Either omission would
contradict $v\eqid U_n$.

Finally, if $b_i$ were absent, assign $a$ to
$y_i,y_{i+1}$ and $1$ to all other variables in $H$.
The word $b_i$ has value $\infty$, while every other listed
word contains at most one of $y_i,y_{i+1}$ and has value
$1$ or $a$. Thus $U_n$ and $v$ would again have different
values. This proves the claim.

Consequently, $v$ is either $U_n$ or $U_n+q_n$.
Both possibilities do occur: the first trivially, and the
second by Lemma~\ref{lem:test-valid}. Therefore the equivalence class has
exactly the two asserted members.
\end{proof}

\begin{remark}
The four conditions in Lemma~\ref{lem:tests} are used as necessary
conditions. We do not need a general converse to that lemma.
For the extra word $q_n$ in the present equivalence class,
sufficiency follows from Lemma~\ref{lem:test-valid}.
\end{remark}

\section{The unbounded-variable obstruction}
\label{sec:nfb}

The next lemma concerns an arbitrary absorption identity valid in
$\W$. Its substitution values may be polynomials, so it applies
to deductions from any proposed basis.

\begin{lemma}[First-step obstruction]\label{lem:step}
Let $n\geq5$ be odd. Suppose that an absorption identity
$s+r\eqid s$ valid in $\W$ has a contextual application
\eqref{eq:step} from $U_n$ to $U_n+q_n$. Then $r$ contains
at least $n+1$ distinct variables.
\end{lemma}

\begin{proof}
We view polynomials as sets of words. Since $q_n\notin U_n$,
the application must be in the direction that adds
$K\varphi(r)$; deleting words cannot introduce $q_n$. Hence
\begin{equation}\label{eq:inclusions}
\begin{aligned}
K\varphi(s)&\subseteq U_n,\\
K\varphi(r)&\subseteq U_n\cup\{q_n\},\\
q_n&\in K\varphi(r).
\end{aligned}
\end{equation}
We establish the required lower bound in three steps.

\emph{Step 1: linearity in the formal variables.}
Every word in the two images in \eqref{eq:inclusions} is
linear. Suppose that a formal variable $v$ occurred at least
twice in a word of $s$ or in $r$. Choose the same nonempty
word of $\varphi(v)$ at each occurrence, choose an arbitrary
word from every other substituted polynomial, and choose a
word from $K$ if that context is present. The resulting
expanded product would belong to one of the two images and
would repeat every letter of the chosen word of $\varphi(v)$.
This contradicts linearity. Thus $r$ and all words of $s$
are linear. Notice that this argument allows substitutions
by sums: the repeated choice is one of the products appearing
in the distributive expansion.

Suppose next that $r$ has just one formal variable $v$.
Applying condition~(3) of Lemma~\ref{lem:tests} to the valid
inequality $r\preceq s$ gives a word of $s$ with content
contained in $\{v\}$. It is nonempty and linear, hence is
$v$ itself. We would then have
$K\varphi(r)\subseteq K\varphi(s)\subseteq U_n$, contrary
to $q_n\in K\varphi(r)$. Therefore
\begin{equation}\label{eq:atleasttwo}
|\con(r)|\geq2.
\end{equation}

\emph{Step 2: a partition of the target word.}
Choose a word $h$ of $K$, or the empty word when the
multiplicative context is absent, and choose
$w_v\in\varphi(v)$ for each $v\in\con(r)$ so that
\begin{equation}\label{eq:partition}
q_n=h\prod_{v\in\con(r)}w_v.
\end{equation}
Such choices exist by the last inclusion in
\eqref{eq:inclusions} and the linearity of $r$.
Because $q_n=x_1\cdots x_nz$ is linear, all chosen factors
are linear, their contents are pairwise disjoint, and those
contents partition $\{x_1,\ldots,x_n,z\}$. Every $w_v$
is nonempty.

For distinct $v,w\in\con(r)$,
Lemma~\ref{lem:tests}(2) gives a word $d$ of $s$ containing
both $v$ and $w$. In its substituted image choose the already
selected words $w_v,w_w$, and choose arbitrary words for the
other variables of $d$. Multiplication by the selected $h$
produces a word of $K\varphi(s)$ and hence a word of $U_n$.
In particular, some word of $U_n$ contains
\begin{equation}\label{eq:pair-witness}
\con(h)\cup\con(w_v)\cup\con(w_w).
\end{equation}
The key feature of $U_n$ is that
\begin{equation}\label{eq:one-x}
\text{each word of $U_n$ containing $z$ contains exactly one $x_i$.}
\end{equation}
Indeed, these words are precisely $a_n,b_1,\ldots,b_n$.

\emph{Step 3: the context is empty and all blocks are singletons.}
If $h$ contained $z$, choose two distinct variables $v,w$
of $r$, using \eqref{eq:atleasttwo}. Their nonempty blocks
$w_v,w_w$ would consist of $x$-variables and, by disjointness,
would contain two different $x$-variables. The word supplied
by \eqref{eq:pair-witness} would contain both of them and
$z$, contrary to \eqref{eq:one-x}.

If $h$ were nonempty but did not contain $z$, it would
contain an $x$-variable. Let $v$ be the unique formal
variable whose selected block contains $z$, and choose
$w\neq v$ in $\con(r)$. The nonempty block $w_w$ contains
another $x$-variable, different from those in $h$. Again
\eqref{eq:pair-witness} contradicts \eqref{eq:one-x}.
These two possibilities exhaust all nonempty $h$, so $h$
is empty.

Let $w_v$ be the block containing $z$. If it also contained
an $x$-variable, pairing it with any other nonempty block
would give two $x$-variables together with $z$ in a word
of $U_n$. Thus $w_v=z$. If any other block contained two
or more $x$-variables, pairing that block with $w_v=z$
would yield the same contradiction. Every remaining block
is therefore a single $x_i$.

The partition \eqref{eq:partition} has exactly the $n+1$
singleton blocks $x_1,\ldots,x_n,z$, one for each distinct
formal variable of $r$. In particular, $r$ contains at
least $n+1$ distinct variables, as required.
\end{proof}

\begin{proposition}\label{prop:bounded}
Let $\Sigma$ be a set of identities valid in $\W$, each
involving at most $k$ variables. If $n\geq5$ is odd and
$n+1>k$, then $\sigma_n$ cannot be deduced from $\Sigma$
together with the axioms of commutative ai-semirings.
\end{proposition}

\begin{proof}
Replace the identities of $\Sigma$ by the absorption
identities in Lemma~\ref{lem:absorption}. Each replacement
is valid in $\W$, and no identity gains a variable.
Suppose that there were a deduction
\[
U_n=T_0,\ T_1,\ \ldots,\ T_m=U_n+q_n
\]
between normalized polynomials. Since every step uses an
identity valid in $\W$, every $T_i$ is equivalent to
$U_n$ in $\W$. Lemma~\ref{lem:isolation} therefore gives
$T_i\in\{U_n,U_n+q_n\}$ for every $i$.

Let $j$ be the least index with $T_j\neq U_n$. It exists
because the last polynomial differs from the first.
Then $T_{j-1}=U_n$ and $T_j=U_n+q_n$, so this single
step is an application of the form in Lemma~\ref{lem:step}.
The word on the absorbed side of its identity must contain
at least $n+1$ variables. The entire identity would then
have more than $k$ variables, a contradiction. Steps
using only the commutative ai-semiring axioms do not
change the normalized polynomial and cannot supply this
first nontrivial step.
\end{proof}

\begin{corollary}\label{cor:small-product}
The variety $\W$ and the $18$-element semiring $H\times D\times G$
are nonfinitely based, although each of $H,D,G$ is finitely based.
\end{corollary}

\begin{proof}
If $\Sigma$ were a basis for $\W$ with a finite variable bound $k$,
then it would derive every $\sigma_n$ by
Lemma~\ref{lem:test-valid}. For odd $n\geq5$ with $n+1>k$,
this contradicts Proposition~\ref{prop:bounded}.
Thus $\W$ has no bounded-variable basis, and in particular no
finite basis.

The direct product belongs to $\W$, and its three projections
are onto the factors. Hence $\Var(H\times D\times G)=\W$.
The factors have $3,2,3$ elements, respectively. Neither
$H$ nor $G$ is isomorphic to $S_7$: $H$ has chain addition,
while $G$ has a two-element subgroup. Their finite basability,
and that of $D$, follow from~\cite[Corollary~5.1]{JRZ}.
The nonfinite basis argument for their product uses
Lemma~\ref{lem:test-valid} and the support arguments above.
\end{proof}

\begin{proof}[Proof of Theorem~\ref{thm:main} for $\Athree$]
Suppose that $\Athree$ had an identity basis $\Sigma$
with a uniform finite variable bound $k$. Every identity
of $\Sigma$ is valid in $\W$ by \eqref{eq:testsinside}.
Adjoin multiplicative commutativity and work in the free
commutative ai-semiring. This introduces only an auxiliary
axiom valid in $\W$; we are not asserting that multiplication
is commutative throughout $\Athree$.

By Proposition~\ref{prop:valid}, every $\sigma_n$ holds
in $\Athree$, and hence follows from the identity basis
$\Sigma$ by equational completeness. The same deduction
remains available after adjoining the auxiliary axioms.
Choose an odd $n\geq5$ with $n+1>k$.
Proposition~\ref{prop:bounded} excludes the resulting
deduction of $\sigma_n$, a contradiction. Thus no such
bounded-variable basis exists. A finite basis would have
a largest variable count, so $\Athree$ is nonfinitely based.
\end{proof}

\section{Passing to arbitrary addition}
\label{sec:arbitrary}

We now transfer the valid identity family to $\Vthree$.
There are two stages: taking an additively idempotent
image, and then taking a quotient with commutative
addition.

We first establish the semigroup fact needed for the first stage.

\begin{lemma}\label{lem:sixth-power}
In every semigroup $S$ of order at most three,
\begin{equation}\label{eq:power}
(ab)^6=a^6b^6\qquad(a,b\in S).
\end{equation}
Moreover, $x^{12}=x^6=x^{36}$ for every $x\in S$.
\end{lemma}

\begin{proof}
The cyclic subsemigroup generated by $x$ has at most three elements.
Its index is at most three and its period is one of $1,2,3$.
The powers $x^6,x^{12},x^{36}$ are therefore equal. In particular,
$x^6$ is idempotent. Put $\rho(x)=x^6$, and let $E$ be the set
of idempotents of $S$. We prove that $\rho$ is a homomorphism.

If $E=\{e\}$, then $\rho$ is the constant map with value $e$,
so the assertion follows from $e^2=e$. If $E=S$, then $\rho$
is the identity map. The only remaining possibility is
$S=\{e,f,t\}$, where $e,f$ are distinct idempotents and $t$
is not idempotent. Rename $e,f$ so that $t^6=e$.
Since $t^2\ne t$, the element $t^2$ is idempotent, and
\[
t^2=(t^2)^3=t^6=e.
\]
Also $te=et=t^3$. This value cannot be $f$, because
$(t^3)^2=t^6=e\ne f=f^2$. Thus $te=et\in\{t,e\}$.

We first check that $E$ is closed under multiplication.
If $ef=t$, then $tf=(ef)f=ef=t$, whence
\[
ef=t^2f=t(tf)=t^2=e,
\]
a contradiction. If $fe=t$, then $ft=f(fe)=fe=t$, and
\[
fe=ft^2=(ft)t=t^2=e,
\]
again a contradiction. Consequently $ef,fe\in E$.

The map $\rho$ fixes $e,f$ and sends $t$ to $e$.
It preserves products of elements of $E$, by the closure just
proved. It also preserves the products $tt,te,et$, since their
images are $e$. To handle $tf$, consider its three possible
values. If $tf=t$, then
\[
ef=t^2f=t(tf)=t^2=e=\rho(tf).
\]
If $tf=e$, then $ef=(tf)f=tf=e=\rho(tf)$.
If $tf=f$, then $ef=t^2f=t(tf)=tf=f=\rho(tf)$.
Thus in every case
$\rho(tf)=ef=\rho(t)\rho(f)$. The same calculation with
the multiplication reversed gives
$\rho(ft)=fe=\rho(f)\rho(t)$. This exhausts all products
and proves~\eqref{eq:power}.
\end{proof}

For $m\geq1$, let $mx$ denote the sum of $m$ copies of $x$.
Applying these facts to the additive reduct of each
three-element semiring, and then using preservation
under HSP, gives
\begin{equation}\label{eq:reflection}
\begin{aligned}
6(x+y)&\eqid6x+6y, &
12x&\eqid6x,\\
(6x)(6y)&\eqid6(xy),&
36x&\eqid6x
\end{aligned}
\end{equation}
throughout $\Vthree$. For the multiplicative equality,
both distributive laws expand $(6x)(6y)$ into a sum of
$36$ copies of $xy$, and $36(xy)\eqid6(xy)$ then applies.
No interchange of unequal summands is needed.

\begin{lemma}\label{lem:reflection}
For $A\in\Vthree$, the map $e_A:A\to A$ given by
$e_A(a)=6a$ is an idempotent endomorphism. Its image
$E(A)$ is an additively idempotent semiring.
If $B$ is any additively idempotent semiring, the relation
\begin{equation}\label{eq:D}
a\mathrel{\mathcal D_+}b
\quad\Longleftrightarrow\quad
a+b+a=a\ \text{ and }\ b+a+b=b
\end{equation}
is a semiring congruence, and $Q(B)=B/\mathcal D_+$ is
an ai-semiring.
\end{lemma}

\begin{proof}
The first and third identities in \eqref{eq:reflection}
give, respectively,
\[
e_A(a+b)=e_A(a)+e_A(b),\qquad
e_A(ab)=e_A(a)e_A(b).
\]
Thus $e_A$ is an endomorphism and its image is a
subsemiring. The other two identities give
\[
e_A(e_A(a))=36a=6a=e_A(a),\qquad
e_A(a)+e_A(a)=12a=6a=e_A(a).
\]
Consequently $e_A$ is a retraction onto $E(A)$,
whose addition is idempotent.

For the second assertion, the additive reduct $(B,+)$
is a band. The semilattice decomposition theorem for
bands states that \eqref{eq:D} is the least semilattice
congruence of this reduct; see~\cite{CP}. In particular,
it is an equivalence relation compatible with addition,
and the additive quotient is commutative and idempotent.
It remains to check compatibility with multiplication.

Suppose $a\mathrel{\mathcal D_+}b$ and $c\in B$.
Left multiplication and distributivity give
\[
ca+cb+ca=c(a+b+a)=ca,\qquad
cb+ca+cb=c(b+a+b)=cb.
\]
Hence $ca\mathrel{\mathcal D_+}cb$.
Right multiplication gives, in the same way,
\[
ac+bc+ac=(a+b+a)c=ac,\qquad
bc+ac+bc=(b+a+b)c=bc,
\]
so $ac\mathrel{\mathcal D_+}bc$.
These two compatibilities, together with transitivity,
give compatibility when both factors are replaced.
Thus $\mathcal D_+$ is a congruence of the full semiring.
The quotient inherits associativity and distributivity,
and its addition is a semilattice operation. It is
therefore an ai-semiring.
\end{proof}

\begin{proposition}\label{prop:translation}
For semiring terms $u,q$,
\[
\Athree\models q\preceq u
\quad\Longleftrightarrow\quad
\Vthree\models 6u+6q+6u\eqid6u.
\]
\end{proposition}

\begin{proof}
Assume first that $\Athree\models q\preceq u$, and fix
a three-element semiring $S$. The ai-semiring
$Q(E(S))$ has at most three elements.
Any ai-semiring of size less than three embeds in one
of size three: adjoin a new element absorbing for both
operations, and repeat if necessary. At each stage
the verification of the axioms is as in
Section~\ref{sec:tests}. Thus $Q(E(S))$ belongs to
$\Athree$.

Let $\theta$ be an arbitrary assignment of variables
in $S$, and put
\[
a=e_S(u^S(\theta))=6u^S(\theta),\qquad
b=e_S(q^S(\theta))=6q^S(\theta).
\]
The composite
$S\xrightarrow{e_S}E(S)\xrightarrow{\pi}Q(E(S))$
is a homomorphism. Term evaluation therefore commutes
with this composite: under the assignment
$x\mapsto\pi(e_S(\theta(x)))$, the values of $u,q$
are $[a],[b]$, respectively. Since $q\preceq u$
holds in $Q(E(S))$, we have
\[
[a]+[b]=[a],
\quad\text{and hence}\quad
(a+b)\mathrel{\mathcal D_+}a.
\]
One of the defining equations of this relation is
$a+(a+b)+a=a$. As addition in $E(S)$ is associative
and idempotent, this reduces to $a+b+a=a$.
By the definitions of $a,b$, it is precisely the
evaluation of $6u+6q+6u\eqid6u$ under $\theta$.
The assignment and $S$ were arbitrary, so the
sandwich identity holds on all three-element
generators and hence throughout $\Vthree$.

Conversely, suppose that the sandwich identity holds
in $\Vthree$. It holds in every three-element
ai-semiring, since these are among the generators
of $\Vthree$. In an ai-semiring $6a=a$ for every
element, and
$u+q+u=u+q$ by commutativity and idempotence.
Thus the sandwich identity reduces to $u+q\eqid u$.
It follows that $\Athree\models q\preceq u$.
\end{proof}

\begin{corollary}\label{cor:intersection}
The ai-subvariety of $\Vthree$ is precisely $\Athree$:
\[
\Vthree\cap\AI=\Athree.
\]
\end{corollary}

\begin{proof}
The inclusion $\Athree\subseteq\Vthree\cap\AI$ is immediate.
For the reverse inclusion, let $A\in\Vthree\cap\AI$, and let
$u\eqid v$ be any identity of $\Athree$. Since addition is a
semilattice operation in $\Athree$, both
$u+v\eqid u$ and $v+u\eqid v$ hold there.
Proposition~\ref{prop:translation} gives
\[
6u+6v+6u\eqid6u,\qquad
6v+6u+6v\eqid6v
\]
in $\Vthree$, and hence in $A$. Because $A$ is an ai-semiring,
these identities reduce to $u+v\eqid u$ and $u+v\eqid v$.
Thus $A$ satisfies $u\eqid v$. It satisfies every identity of
$\Athree$, so $A\in\Athree$ by Birkhoff's theorem.
\end{proof}

By Propositions~\ref{prop:valid} and~\ref{prop:translation},
$\Vthree$ satisfies
\begin{equation}\label{eq:fullfamily}
6U_n+6q_n+6U_n\eqid6U_n
\qquad(n\geq3\text{ odd}).
\end{equation}
The structural consequence above also gives a short completion
of the main proof.

\begin{proof}[Completion of the proof of Theorem~\ref{thm:main}]
Suppose that $\Sigma$ were an identity basis for
$\Vthree$ whose identities involve at most $k$
variables. By Corollary~\ref{cor:intersection}, the set
\[
\Sigma\cup\{x+y\eqid y+x,\ x+x\eqid x\}
\]
is a basis for $\Athree$ whose identities involve at most
$\max(k,2)$ variables. This contradicts the result already
proved for $\Athree$. Hence $\Vthree$ also has no bounded-variable
basis, and Theorem~\ref{thm:main} follows.
\end{proof}

\begin{remark}
The argument applies even when the proposed basis is infinite.
The family~\eqref{eq:fullfamily} records explicit consequences of
$\Vthree$; it is not asserted to be a complete identity basis.
\end{remark}

\smallskip
\noindent\textbf{Acknowledgments.}
This work was supported by the National Natural Science Foundation of China (12371024), the Science and Technology Research Program of Chongqing Municipal Education Commission (KJZD-K2024011102) and the Chongqing Natural Science Foundation Innovation and Development Joint Fund (Municipal Education Commission) (CSTB2025NSCQ-LZX0067).

GPT-6 Astra (OpenAI) assisted with language polishing and some of the calculations. The authors take full responsibility for the entire paper.

\end{document}